\documentclass[12pt]{article}
\usepackage{graphicx} 
\usepackage{amsmath, amssymb, amsthm, mathtools, comment}
\usepackage{mathrsfs}
\usepackage{mathabx}
\usepackage[letterpaper,top=3cm,bottom=3cm,left=3cm,right=3cm,marginparwidth=1.75cm]{geometry}
\usepackage[colorlinks=true, allcolors=blue]{hyperref}
\usepackage{bbm}
\usepackage{xcolor}
\usepackage[most]{tcolorbox}

\newtheorem{theorem}{Theorem}[section]
\newtheorem{proposition}[theorem]{Proposition}
\newtheorem{lemma}[theorem]{Lemma}
\newtheorem{definition}[theorem]{Definition}

\newtheorem*{fact*}{Fact}

\newtheorem*{ex*}{Example}

\newtheorem{rem}[theorem]{Remark}

\newtcbtheorem{question}{Question}%
{colback=green!5,colframe=green!35!black,fonttitle=\bfseries}{th}

\newtheorem*{remark}{Remark}

\title{\textsc{Small moments of the sensitivity of polynomial threshold functions}}
\author{\textsc{Chun-Kai Tseng}, \textsc{Alexander Volberg}}
\date{\today}

\begin{document}

\maketitle


\begin{abstract}

In the first version of Chang, Slote, Volberg, and Zhang's paper \cite{BSA_of_PTF}, the authors modify a nice recursive approach due to Kane in \cite{Correct_exponent_for_AS} where he bounded the average sensitivity of polynomial threshold functions. In \cite{BSA_of_PTF} Kane's argument was adopted to estimate the boolean surface area of  polynomial threshold function. The bridge is a combinatorial averaging lemma considering all balanced partitions. The lemma serves as a substitute for an additive property of average sensitivity. With the lemma, one can apply a Kane-type algorithm to derive a recurrence. Solving the recurrence then gives an upper bound of $e^{C_d \sqrt{\log n}}$ for the boolean surface area. 

In the second version of the same paper, the authors derive a polylog upper bound for BSA of PTFs. The difference is that they use a tail estimate for the sensitivity function. With the help of a polynomial restriction lemma in \cite{poly_restriction} they sharpen the upper bound. It is noteworthy that when applying the polynomial restriction, each coordinate is put into each part independently with equal probability. As a result, a partition does not necessarily have equal-size blocks. In other words, it may not be balanced.

In this note, we first investigate the effect of different partitioning. Second, we use the recursive method in the first version to derive a polylog upper bound for $\mathbb E[s(x)^{\eta}]$ where $\eta < 1/2$. It is interesting to note the phase transition that happens at $\eta=1/2$ in both versions of the proof (but in a completely different  form). Section \ref{PhaseTr-s} treats this.

\end{abstract}

\section{Proving Tail Estimate Using Balanced Partitioning}
\label{balance}

In this section, we prove the same tail estimate in the second version of \cite{BSA_of_PTF} using balanced partitioning. In fact, there are not many differences.

\begin{proposition}[The Same Tail Estimate: Balanced-block Partition inside the proof]
Let \(f : \{-1,1\}^n \to \{-1,1\}\) be a degree-\(d\) PTF. Assume \(n \ge 256\), and let \(m\) be an integer satisfying
\[
16 \le m \le n,
\qquad m \mid n.
\]
Then
\[
\Pr[s_f(x)\ge m]
\le
\frac{8(\log(en))^{2K_d}}{\sqrt m},
\]
where $K_d$ is a constant depending on $d$.
\end{proposition}

\begin{proof}
Fix \(m \in \{16,\dots,n\}\) with \(m \mid n\). Sample
\(x \in \{-1,1\}^n\) uniformly at random. Independently, sample a uniformly random
partition $\Pi$
\[
[n] = G_1 \sqcup \cdots \sqcup G_m, \ \Pi=(G_1, \cdots, G_m)
\]
into \(m\) blocks of equal size
\[
|G_\ell| = \frac{n}{m}.
\]
Then choose \(L\) uniformly from \([m]\).

From \((x,\Pi,L)\), form the restriction \(\rho_{x,\Pi,L}\) by
\[
(\rho_{x,\Pi,L})_i
=
\begin{cases}
\ast, & i\in G_L,\\
x_i, & i\notin G_L.
\end{cases}
\]
Thus the live coordinates are precisely those belonging to the random block \(G_L\).

Define the events
\[
E := \{x : s_f(x)\ge m\},
\]
and
\[
B :=
\left\{
(x,\Pi,L):
s_{f_{\rho_{x,\Pi,L}}}(x_{G_L})\ge1
\right\}.
\]

Fix \(x\in E\), and define the sensitive set
\[
S(x)
:=
\{i\in[n]: f(x)\neq f(x^{\oplus i})\}.
\]
Then
\[
|S(x)|=s_f(x)\ge m.
\]

We claim that
\[
\Pr_{\Pi,L}[G_L\cap S(x)\neq\varnothing]
\ge
1-\frac1e.
\]
Indeed, since \(G_L\) is a uniformly random subset of \([n]\) of cardinality \(n/m\),
\[
\Pr_{\Pi,L}[G_L\cap S(x)=\varnothing]
=
\frac{
\binom{n-|S(x)|}{n/m}
}{
\binom{n}{n/m}
}.
\]
Since \(|S(x)|\ge m\),
\[
\Pr_{\Pi,L}[G_L\cap S(x)=\varnothing]
\le
\frac{
\binom{n-m}{n/m}
}{
\binom{n}{n/m}
}.
\]
Expanding the ratio,
\[
\frac{
\binom{n-m}{n/m}
}{
\binom{n}{n/m}
}
=
\prod_{j=0}^{n/m-1}
\left(
1-\frac{m}{n-j}
\right)
\le
\left(1-\frac{m}{n}\right)^{n/m}
\le
e^{-1}.
\]
Hence
\[
\Pr_{\Pi,L}[G_L\cap S(x)\neq\varnothing]
\ge
1-\frac1e
>
\frac12.
\]

Now suppose
\[
i\in G_L\cap S(x).
\]
Since \(i\) remains live under the restriction \(\rho_{x,\Pi,L}\), the point
\(x_{G_L}\) in the restricted cube corresponds to the original point \(x\), and flipping
the live coordinate \(i\) sends \(x_{G_L}\) to the restricted point corresponding to
\(x^{\oplus i}\), while all coordinates outside \(G_L\) remain fixed. Therefore
\[
f_{\rho_{x,\Pi,L}}(x_{G_L})
=
f(x)
\neq
f(x^{\oplus i})
=
f_{\rho_{x,\Pi,L}}
\bigl((x_{G_L})^{\oplus i}\bigr),
\]
so
\[
s_{f_{\rho_{x,\Pi,L}}}(x_{G_L})\ge1.
\]
Consequently,
\[
\Pr_{\Pi,L}[B\mid x]
\ge
\Pr_{\Pi,L}[G_L\cap S(x)\neq\varnothing]
>
\frac12, \quad x\in E\,.
\]

Averaging over \(x\in E\),
\[
\Pr[B]
=
\mathbb E_x
\Pr_{\Pi,L}[B\mid x]
\ge
\frac12
\Pr[E].
\]

Condition on a restriction \(\rho=\rho_{x,\Pi,L}\). The restricted function
\(f_\rho\) lives on exactly
\[
\ell=\frac{n}{m}
\]
variables. If \(f_\rho\) is \(\delta\)-close to a constant, then by the same argument as in the original proof we have
\[
\Pr[B\mid \rho]
=
\Pr_{y\in\{-1,1\}^{\ell}}
[s_{f_\rho}(y)\ge1]
\le
(\ell+1)\delta
=
\left(\frac{n}{m}+1\right)\delta.
\]
Hence
\[
\Pr[B]
\le
\Pr_{\rho}
[f_\rho \text{ is not }\delta\text{-close to a constant}]
+
\delta\left(\frac{n}{m}+1\right).
\]

We now apply the block restriction lemma. First fix the balanced partition
\[
\Pi=(G_1,\dots,G_m).
\]
Conditioned on this \(\Pi\), the following procedure is exactly the random block
restriction associated with the partition \(\Pi\): choose \(L\in[m]\) uniformly, leave
all coordinates in \(G_L\) free, and assign every coordinate outside \(G_L\) an
independent uniform sign.

Indeed, in our construction the outside coordinates are fixed according to the random
point \(x\). Since \(x\) is uniformly distributed on \(\{-1,1\}^n\), the vector
\[
x_{[n]\setminus G_L}
\]
is an independent uniform assignment to the variables outside \(G_L\). Therefore,
conditional on \(\Pi\), the restriction \(\rho_{x,\Pi,L}\) has exactly the distribution
of a block restriction with respect to the partition \(\Pi\).

Hence, for every fixed balanced partition \(\Pi\), the block restriction lemma gives
\[
\Pr_{x,L}
\left[
f_{\rho_{x,\Pi,L}}
\text{ is not }\delta\text{-close to a constant}
\,\middle|\, \Pi
\right]
\le
\left(
\frac1{\sqrt m}
+
\delta
\right)
\left(
\log m\cdot \log\frac1\delta
\right)^{K_d}.
\]
Here the probability is only over \(x\) and \(L\), with \(\Pi\) fixed.

Since the right-hand side is independent of the particular partition \(\Pi\), we may
average over the random choice of \(\Pi\). Thus
\[
\Pr_{x,\Pi,L}
\left[
f_{\rho_{x,\Pi,L}}
\text{ is not }\delta\text{-close to a constant}
\right]
=
\mathbb E_\Pi
\Pr_{x,L}
\left[
f_{\rho_{x,\Pi,L}}
\text{ is not }\delta\text{-close to a constant}
\,\middle|\, \Pi
\right]
\]
\[
\le
\left(
\frac1{\sqrt m}
+
\delta
\right)
\left(
\log m\cdot \log\frac1\delta
\right)^{K_d}.
\]
Substituting \(\delta=\sqrt m/n\),
\[
\Pr[B]
\le
\left(
\frac1{\sqrt m}
+
\frac{\sqrt m}{n}
\right)
\left(
\log m\cdot \log\frac{n}{\sqrt m}
\right)^{K_d}
+
\frac1{\sqrt m}
+
\frac{\sqrt m}{n}.
\]
Since \(m\le n\),
\[
\frac{\sqrt m}{n}\le\frac1{\sqrt m},
\]
and therefore
\[
\Pr[B]
\le
\frac{4}{\sqrt m}
(\log(en))^{2K_d}.
\]
Finally, since
\[
\Pr[E]\le2\Pr[B],
\]
we conclude
\[
\Pr[s_f(x)\ge m]
=
\Pr[E]
\le
\frac{8(\log(en))^{2K_d}}{\sqrt m}.
\]
This completes the proof.
\end{proof}

\section{Bounding Lower Moment Using Recursive Method}
\label{lowerMoments}

In this section, we give $\mathbb E[s_f (x)^\eta]$ a polylog bound in terms of $n$, where $f\colon \{-1,1\}^n \to \{-1,1\}$ is a polynomial threshold function. For simplicity, we focus on $\eta = 1/4$ since the same mechanism applies to other numbers strictly smaller than $1/2$. Also, we always assume that $n$ is divisible by the number of parts in the partition.

\subsection{Balanced Block Averaging Lemma}

\begin{lemma}[Block averaging for the \(1/4\)-moment] 
Assume \(n=bq\). Let \(y_1,\dots,y_n\in\{0,1\}\), and write
\[
S:=\sum_{j=1}^n y_j.
\]
Let \(\Pi=(G_1,\dots,G_b)\) be a uniformly random partition of \([n]\) into \(b\)
blocks of equal size \(q\). Define
\[
A:=S^{1/4}
\]
and
\[
B:=
\frac{1}{b^{3/4}}
\mathbb E_\Pi
\sum_{\ell=1}^b
\left(\sum_{j\in G_\ell} y_j\right)^{1/4}.
\]
Then
\[
B\le A\le B+C b^{1/4},
\]
where \(C>0\) is an absolute constant.
\end{lemma}

\begin{proof}
For a fixed partition \(\Pi\), put
\[
S_\ell:=\sum_{j\in G_\ell}y_j.
\]
Then
\[
\sum_{\ell=1}^b S_\ell=S.
\]
Since \(t\mapsto t^{1/4}\) is concave, Jensen gives
\[
\frac1b\sum_{\ell=1}^b S_\ell^{1/4}
\le
\left(\frac1b\sum_{\ell=1}^b S_\ell\right)^{1/4}
=
\left(\frac Sb\right)^{1/4}.
\]
Multiplying by \(b^{1/4}\), we get
\[
\frac1{b^{3/4}}\sum_{\ell=1}^b S_\ell^{1/4}
\le
S^{1/4}.
\]
Averaging over \(\Pi\) proves
\[
B\le A.
\]

We now prove the reverse estimate. By symmetry, for each fixed \(\ell\), $S_\ell$ has the same hypergeometric distribution
\[
X\sim Hg(n,S,q).
\]
Therefore
\[
B
=
b^{1/4}\mathbb E X^{1/4}.
\]
Also
\[
\mathbb E X=\frac{qS}{n}=\frac Sb.
\]
Set
\[
\mu:=\mathbb E X=\frac Sb.
\]

We shall use the elementary estimate
\[
\mathbb E X^{1/4}
\ge
\mu^{1/4}
-
C\mu^{-7/4}\text{Var}(X),
\]
valid for any nonnegative random variable \(X\) with \(\mu=\mathbb E X>0\).
Indeed, the inequality follows from the pointwise bound
\[
t^{1/4}
\ge
1+\frac14(t-1)-C(t-1)^2,
\qquad t\ge0,
\]
applied to \(t=X/\mu\) and then averaged.

Hence
\[
A-B
=
S^{1/4}-b^{1/4}\mathbb E X^{1/4}
=
b^{1/4}\left(\mu^{1/4}-\mathbb E X^{1/4}\right)
\le
C b^{1/4}\mu^{-7/4}\text{Var}(X).
\]
For a hypergeometric random variable,
\[
\text{Var}(X)
=
\frac{qS(n-S)(n-q)}{n^2(n-1)}
\le
\frac{qS}{n}
=
\mu.
\]
Therefore
\[
A-B
\le
C b^{1/4}\mu^{-3/4}
=
C b^{1/4}\left(\frac Sb\right)^{-3/4}
=
C\frac{b}{S^{3/4}}.
\]

If \(S\ge b\), then
\[
\frac{b}{S^{3/4}}\le b^{1/4},
\]
so
\[
A-B\le Cb^{1/4}.
\]
If \(S<b\), then simply
\[
A=S^{1/4}<b^{1/4}\le B+b^{1/4}.
\]
Thus in all cases,
\[
A\le B+Cb^{1/4}.
\]
This proves the lemma.
\end{proof}

\begin{rem} \label{Block averaging is sharp}
If one applies the same block-averaging argument to the \(1/2\)-moment,
which corresponds to the BSA case, then one obtains
\[
B \le A \le B+C b^{1/2},
\]
where
\[
A:=S^{1/2},
\qquad
B:=
\frac{1}{b^{1/2}}
\mathbb E_\Pi
\sum_{\ell=1}^b
\left(\sum_{j\in G_\ell} y_j\right)^{1/2}.
\]
Compared with Proposition~3.1 in \cite{BSA_of_PTF}, where the error term is
of order \(b\), this gives the smaller error \(C b^{1/2}\).

Moreover, the powers \(b^{1/2}\) in the BSA case and \(b^{1/4}\) in the
\(1/4\)-moment case are sharp. We illustrate this for the \(1/4\)-moment.
Take
\[
S=q=b,
\qquad
n=b^2.
\]
Recall that
\[
S_\ell:=\sum_{j\in G_\ell} y_j .
\]
Then, for each fixed \(\ell\),
\[
S_\ell \sim \operatorname{Hyp}(b^2,b,b).
\]
Let \(Y\sim \operatorname{Pois}(1)\). Since
\[
\operatorname{Hg}(b^2,b,b)\Longrightarrow \operatorname{Pois}(1)
\]
as \(b\to\infty\), we have
\[
\mathbb E S_\ell^{1/4}\longrightarrow \mathbb E Y^{1/4}.
\]
By the strict concavity of \(t\mapsto t^{1/4}\),
\[
0<\mathbb E Y^{1/4}<(\mathbb E Y)^{1/4}=1.
\]
Therefore,
\[
A-B
=
b^{1/4}
-
b^{1/4}\mathbb E S_\ell^{1/4}
=
b^{1/4}\left(1-\mathbb E S_\ell^{1/4}\right)
\sim
b^{1/4}\left(1-\mathbb E Y^{1/4}\right).
\]
Hence the error is bounded from below by \(c b^{1/4}\) for some absolute
constant \(c>0\) along this sequence of examples. The same construction also
shows that the \(b^{1/2}\) error is sharp in the BSA case, and more generally
gives the corresponding \(b^\eta\) lower bound for every \(0<\eta<1\).
\end{rem}

\subsection{Building the Recurrence}

For convenience, we introduce the following definitions:
\begin{definition}
\[
M_{1/4}[f]:=\mathbb E_x [s_f(x)^{1/4}].
\]
For \(a>0\), define
\[
M_{1/4}(d,n,a)
\]
to be the supremum of \(M_{1/4}[\text{sgn}(p)]\) over all degree-\(d\) polynomials
\(p\) on \(\{-1,1\}^n\) satisfying
\[
\alpha(p)\le a.
\]
Similarly, define \(MR_{1/4}(d,n,a,\tau)\) by adding the assumption that \(p\)
is \(\tau\)-regular.
\end{definition}

The analogue of the block estimate for \(BSA\) is
\[
M_{1/4}[f]
\le
\frac1{b^{3/4}}
\mathbb E_\Pi
\sum_{\ell=1}^b
\mathbb E_{A_\ell}
M_{1/4}[f_{A_\ell}]
+
C b^{1/4}.
\]
Indeed, this follows from applying the combinatorial estimate 
\[
\left(\sum_{i=1}^n y_i\right)^{1/4}
\le
\frac1{b^{3/4}}
\mathbb E_\Pi
\sum_{\ell=1}^b
\left(\sum_{i\in G_\ell} y_i\right)^{1/4}
+
C b^{1/4},
\]
applied to
\[
y_i(x)=\mathbf 1_{\{f(x)\neq f(x^{\oplus i})\}}
\]
at each point $x\in \{-1,1\}^n$.

Therefore, arguing exactly as in \cite{BSA_of_PTF}, we obtain \begin{equation} \label{Regular->Smaller scale} MR_{1/4}(d,n,a,\tau) \le C b^{1/4} + b^{1/4} \operatorname*{sup}_{\aleph} \mathbb E_{\aleph} M_{1/4}\bigl(d,n/b,\aleph\bigr), \end{equation} where the supremum is taken over all admissible non-negative random variables \(\aleph\) taking only finitely many values and satisfying 
\[ 
\mathbb E\aleph \le C_d\left(a b^{-1/2}+\tau^{1/(8d)}\right). 
\] 
Next, the regularization step gives \begin{equation} \label{normal->regular} M_{1/4}(d,n,a) 
\le 
D^{1/4} + 3(n\varepsilon)^{1/4} + \operatorname*{sup}_{\tilde\aleph} \mathbb E_{\tilde \aleph} MR_{1/4}(d,n,\tilde \aleph,\tau), \end{equation} 
where 
\[ 
D=\tau^{-1} \left(d\log(1/\tau)\log(1/\varepsilon)\right)^{O(d)} 
\] 
and where the supremum is taken over all admissible non-negative random variables \(\tilde\aleph\) taking only finitely many values and satisfying 
\begin{equation} \label{lea} 
\mathbb E\tilde \aleph\le a. \end{equation} 
Combining \eqref{Regular->Smaller scale} and \eqref{normal->regular}, we get \begin{equation} \label{4} 
M_{1/4}(d,n,a) \le D^{1/4} + 3(n\varepsilon)^{1/4} + C b^{1/4} + b^{1/4} \operatorname*{sup}_{\aleph} \mathbb E_{\aleph} M_{1/4}\bigl(d,n/b,\aleph\bigr), \end{equation} 
where 
\begin{equation} \label{Etaleph} 
\mathbb E\aleph \le C_d\left(a b^{-1/2}+\tau^{1/(8d)}\right). \end{equation} 

Let us explain the last passage. Combining \eqref{Regular->Smaller scale} and \eqref{normal->regular}, the main term is initially of the form \[ \operatorname*{sup}_{\tilde\aleph} \mathbb E_{\tilde\aleph} \left[ \operatorname*{sup}_{\aleph=\aleph(\tilde\aleph)} \mathbb E_{\aleph} M_{1/4}\bigl(d,n/b,\aleph\bigr) \right], \] where the outer supremum is taken over all non-negative random variables \(\tilde\aleph\) taking finitely many values and satisfying \[ \mathbb E\tilde\aleph\le a, \] and, for each realized value of \(\tilde\aleph\), the inner supremum is taken over all non-negative random variables \(\aleph=\aleph(\tilde\aleph)\) taking finitely many values and satisfying \[ \mathbb E[\aleph\mid \tilde\aleph] \le C_d\left(\tilde\aleph b^{-1/2}+\tau^{1/(8d)}\right). \] Strictly speaking, the inner supremum need not be attained. However, since \(\tilde\aleph\) takes finitely many values, the expectation over \(\tilde\aleph\) is a finite sum over its possible values. Hence, by the usual \(\varepsilon\)-principle, for each possible value of \(\tilde\aleph\) we may choose an \(\gamma\)-optimal random variable \(\aleph_0=\aleph_0(\tilde\aleph)\). Thus, up to an arbitrarily small \(\gamma\)-loss, we may write \[ \begin{aligned} 
& \operatorname*{sup}_{\tilde\aleph} \mathbb E_{\tilde\aleph} \left[ \operatorname*{sup}_{\aleph=\aleph(\tilde\aleph)} \mathbb E_{\aleph} M_{1/4}\bigl(d,n/b,\aleph\bigr) \right] \\ 
&\le \operatorname*{sup}_{\tilde\aleph} \left\{  \mathbb E_{\tilde\aleph} \left[ \mathbb E_{\aleph_0( \tilde\aleph)} 
\left[ M_{1/4}\bigl(d,n/b,\aleph_0\bigr) \right] \right]\right\} +\gamma. \end{aligned} \] 
After making these choices, we regard the two-step randomness as producing one single non-negative random variable taking finitely many values, still denoted by \(\aleph_0\). Its expectation satisfies \[ \mathbb E\aleph_0 = \mathbb E_{\tilde\aleph} \mathbb E[\aleph_0 \mid \tilde\aleph] \le C_d\left( b^{-1/2}\mathbb E\tilde\aleph+\tau^{1/(8d)} \right) \le C_d\left( a b^{-1/2}+\tau^{1/(8d)} \right). \] 
Therefore, \[ \begin{aligned} & \operatorname*{sup}_{\tilde\aleph} \left\{ \mathbb E_{\tilde\aleph} \left[ \mathbb E_{\aleph_0 \mid \tilde\aleph} \left[ M_{1/4}\bigl(d,n/b,\aleph_0\bigr) \right] \right] \right\} +\gamma \\ 
&= \operatorname*{sup}_{\tilde\aleph} \left\{ \mathbb E_{\aleph_0} \left[ M_{1/4}\bigl(d,n/b,\aleph_0\bigr) \right] \right\} +\gamma \\ &\le \operatorname*{sup}_{\tilde\aleph} \left\{ \operatorname*{sup}_{\aleph} \mathbb E_{\aleph} M_{1/4}\bigl(d,n/b,\aleph\bigr) \right\} +\gamma \\ 
&= \operatorname*{sup}_{\aleph} \mathbb E_{\aleph} M_{1/4}\bigl(d,n/b,\aleph\bigr) +\gamma, 
\end{aligned} 
\] 
where the inner supremum in the second-to-last line and the last supremum are
taken over all non-negative random variables \(\aleph\) taking finitely many
values and satisfying \[ \mathbb E\aleph \le C_d\left(a b^{-1/2}+\tau^{1/(8d)}\right). \] Letting \(\gamma\to0\), we obtain the recurrence \eqref{4} with the constraint \eqref{Etaleph}.

\subsection{Solving the Recurrence}

Take \(\varepsilon=1/n\). Write
$$
F(n,a):=M_{1/4}(d,n,a).
$$
From the previous subsection, we have

\begin{equation}
\label{Pn}
F(n,a)
\le
P(n)+C b^{1/4}
+
b^{1/4} \cdot
\sup_{\aleph}\ \mathbb E_\aleph
F\bigl( n/b,\aleph\bigr),
\end{equation}
where
\[
P(n):=
\tau^{-1/4}
\left(d\log(1/\tau)\log n\right)^{O(d)}+3,
\]
and the supremum is taken over all non-negative random variables $\aleph$ taking only finitely many values and satisfying
\[
\mathbb E\aleph
\le
C_d\left(a b^{-1/2}+\tau^{1/(8d)}\right).
\]

We now solve this recurrence. Let
\[
A(n):=(K\log n)^{-C\,d}
\]
be the small-\(\alpha\) stopping scale, chosen so that
\[
a\le A(n)
\quad\Longrightarrow\quad
F(n,a)\le a.
\]
Choose \(b=b_d\) to be a sufficiently large constant {\bf depending only on} $d$, and choose \(\tau=\tau(n)\) by
\begin{equation}
\label{rootProblem}
\tau^{1/(8d)}=A(n)b^{-1/2}.
\end{equation}

Then
\[
b^{1/4}\mathbb E\aleph
\le
C_d\left(a b^{-1/4}+A(n)b^{-1/4}\right).
\]
In the recursive regime \(a>A(n)\), this becomes
\[
b^{1/4}\mathbb E\aleph
\le
2C_d a b^{-1/4}.
\]

We claim that for \(M=M(d)\) sufficiently large,
$$
F(n,a)\le a\Phi(n),
\qquad
\Phi(n):=(\log(en))^M.
$$
We prove this by induction on \(n\). The case \(a\le A(n)\) follows from the stopping
lemma. Hence assume \(a>A(n)\). By the induction hypothesis,
\[
b^{1/4} \cdot \sup_{\aleph} \left\{\mathbb E_\aleph
F\bigl(n/b,\aleph\bigr) \right\}
\le
\sup_{\aleph} \{ b^{1/4} \cdot \mathbb E\aleph\} \, \cdot
\Phi\bigl(n/b \bigr)
\le
2C_d a b^{-1/4} \cdot
\Phi\bigl(n/b \bigr)
.
\]
Thus
\begin{equation} \label{Recurrence for smaller moment}
F(n,a)
\le
P(n)+C b^{1/4}
+
2C_d a b^{-1/4}
\Phi\bigl( n/b \bigr).
\end{equation}
Since \(b=b_d\) is fixed, choosing it sufficiently large gives
\begin{equation} \label{contraction effect}
2C_d b^{-1/4}\le \frac13.
\end{equation}
Also
\[
\Phi\bigl( n/b \bigr)\le \Phi(n).
\]
We get
\[
2C_d b^{-1/4}
\Phi\bigl( n/b \bigr)
\le
\frac13 \Phi(n).
\]
Therefore the recursive term is bounded by
\[
\frac13 a\Phi(n).
\]

It remains to control the first two terms. First, $C b^{1/4}$ is a constant, which can be made less than $a/3 \cdot \Phi(n)$. For $P$, since we have \eqref{rootProblem} 
\[
\tau^{1/(8d)}=A(n)b^{-1/2},
\]
we have
\[
\tau^{-1}
=
A(n)^{-8d}b^{4d}= (\log n)^{Cd\cdot 8d} b^{4d}\,.
\]
Since \(b\) is fixed depending only on $d$, and \(A(n)^{-1}\) is a $Cd$ power of \(\log n\), it follows that
\[
P(n)=\tau^{-1/4}
\left(d\log(1/\tau)\log n\right)^{O(d)}+3 \le (\log(en))^{\tilde{C}_d} \le (\log(en))^{C\, d^2} \,.
\]
Choosing \(M=M(d)=C_1 d^2 +c_2 d\) in the definition of $\Phi$ sufficiently large gives
\[
P(n)
\le (\log(en))^{C\, d^2}
\le \frac13 A(n)\Phi(n).
\]
Since \(a>A(n)\), we obtain
\[
P(n)
\le
\frac13 a\Phi(n).
\]
Combining the estimates,
\[
F(n,a)
\le
\frac13 a\Phi(n)+\frac13 a\Phi(n)+\frac13 a\Phi(n)
=
a\Phi(n).
\]
This closes the induction.

Taking \(a=1\), we conclude that every degree-\(d\) PTF satisfies
\[
\mathbb E_x [s_f(x)^{1/4}]
\le
(\log(en))^{C_1d^2+C_2 d}\,.
\]

\begin{rem}
This is where one acquires $Cd^2$ exponent over $\log n$. But Kane \cite{Correct_exponent_for_AS} does not acquire it. The reason that Kane 14 does
not acquire $Cd^2$  exponent over $\log n$ is in the fact that in the last display formula before Section 5 of Kane \cite{Correct_exponent_for_AS}, the choice of $b$ is drastically different from our choice of $b$.
We choose $b$ depending only on $d$. The drastically different choice of $b$ in Kane 14 is $b=n^{c/d}$. Of course we cannot afford that for the logarithmic estimate, because our estimate 
\eqref{Pn} has
$b^{1/4}$ in the right hand side, so the largest $b$ we can afford is also polylogarithmic in $n$.
\end{rem}

\begin{rem}
The choice of the stopping scale is one of the places where the power
$Cd^2$ enters. If one can manage to choose $A(n)$ better than our choice,
then one may be able to improve this power. However, the stopping scale
$\alpha(p)\lesssim (\log n)^{-d}$ seems to be rather natural from the
standard argument.

Indeed, let $f=\operatorname{sgn}(p)$. In order to stop the recursion for
the $1/4$-moment, it is enough to make
\[
M_{1/4}[f]
=
\mathbb{E}_{x}\bigl[s_f(x)^{1/4}\bigr]
\]
small. Since $s_f(x)$ is integer-valued, we have
\[
s_f(x)^{1/4}\leq s_f(x),
\]
and hence
\[
M_{1/4}[f]
\leq
\mathbb{E}_{x}[s_f(x)]
=
\operatorname{AS}[f].
\]
On the other hand,
\[
\operatorname{AS}[f]
\lesssim
n \cdot \Pr\{f \text{ takes its less common value}\}.
\]
Thus the stopping argument naturally reduces to showing that the less
common value of $f$ occurs with very small probability.

Writing $\mu=\mathbb{E}p$, and using explicitly that
$f=\operatorname{sgn}(p)$, we have
\[
\Pr\{f \text{ takes its less common value}\}
\leq
\Pr\{|p-\mu|>|\mu|\}.
\]
The standard argument used in Kane's stopping lemma gives
\[
\frac{|\mu|}{\|p-\mu\|_2}
\gtrsim
2^{-O(d)}\alpha(p)^{-1/2}.
\]
Equivalently,
\[
|\mu|
=
\frac{|\mu|}{\|p-\mu\|_2}\,\|p-\mu\|_2
\gtrsim
2^{-O(d)}\alpha(p)^{-1/2}\|p-\mu\|_2.
\]
Therefore, by the degree-$d$ concentration inequality,
\[
\begin{aligned}
\Pr\{|p-\mu|>|\mu|\}
&\leq
\Pr\left\{
|p-\mu|
\gtrsim
2^{-O(d)}\alpha(p)^{-1/2}\|p-\mu\|_2
\right\} \\
&\lesssim
\exp\left(
-c\left(2^{-O(d)}\alpha(p)^{-1/2}\right)^{2/d}
\right) \\
&=
\exp\left(
-c\,2^{-O(1)}\alpha(p)^{-1/d}
\right).
\end{aligned}
\]
To make this probability small enough for the stopping argument, one is
therefore naturally led to the condition
\[
\alpha(p)^{-1/d}\gtrsim \log n,
\]
or equivalently
\[
\alpha(p)\lesssim (\log n)^{-d}.
\]
This explains why Kane's stopping scale
\[
\alpha(p)\leq (K\log n)^{-d}
\]
is natural. Consequently, improving the order of the final exponent by
choosing a substantially better $A(n)$ would likely require a new input,
rather than only a more careful optimization of this stopping argument.
\end{rem}

\subsection{What  Does Not Work for the $1/2$-moment in this approach}
Let us see why we cannot obtain a polylog bound for BSA using Kane-type direct proof. Establishing a polylog 
bound means that $\Phi(n)=\text{polylog}(n)$. We try to find $\Phi=\Phi(n)$, $b=b(n),$ and other parameters to complete the induction.

The same argument as above, combined with the \(C b^{1/2}\)-error estimate in
Remark~\ref{Block averaging is sharp}, gives the recurrence
(compare with \eqref{Recurrence for smaller moment})
\[
F(n,a)
\le
P(n)+C b^{1/2}
+
b^{1/2}
\mathbb E_\aleph
F\bigl( n/b,\aleph\bigr),
\]
where
\[
P(n):=
\tau^{-1/2}
\left(d\log(1/\tau)\log n\right)^{O(d)}+3,
\]
The same type of $A(n)$ can be used to get (in the regime $a\ge A(n)$):
$$
F(n,a)
\le
P(n) + Cb^{1/2}
+
2C_d a
\Phi\bigl( n/b \bigr)
$$
What we want to show is
\begin{equation} \label{Recurrence for BSA}
P(n)+b^{1/2}
+
2C_d \cdot a \cdot {\color{red} b^{0}} \cdot
\Phi\bigl( n/b \bigr)
\le a \Phi(n).
\end{equation}
The key is that in the recursion term, there is no decay in $b$.

\medskip

If the above is true, then we must have
$$
2C_d
\Phi\bigl( n/b \bigr) \le \Phi(n).
$$
We rewrite it as
$$
\frac{\Phi(n)}{\Phi(\frac{n}{b})} \ge 2C_d \gg 1.
$$
This inequality imposes a restriction on how fast the upper bound grows. Suppose $\Phi(n) = C'_d (\log n)^{K_d}$. Then the left-hand side gives
$$
\frac{\Phi(n)}{\Phi(\frac{n}{b})}
=
\frac{C'_d (\log n)^{K_d}}{C'_d (\log n-\log b)^{K_d}} 
=
\frac{1}{\left(1-\frac{\log b}{\log n}\right)^{K_d}}
\ge 2C_d
$$
Since $C_d$ is a constant, This implies that
$$
\log b \ge c \log n
$$
for some small $c>0$, which implies that $b \ge n^c$. However, this contradicts \eqref{Recurrence for BSA} since it also implies
$$
n^{c/2} \le b^{1/2} \le a \Phi(n) \le \text{polylog}(n). 
$$

\begin{remark} 
One may hope to obtain a polylogarithmic upper bound for BSA by improving the error term in the analogous block-averaging estimate for the \(1/2\)-moment from \(C b^{1/2}\) to \(\operatorname{polylog}(b)\). If such an improvement were available, then \eqref{Recurrence for BSA} would be replaced by \[ P(n)+\operatorname{polylog}(b) + 2C_d \cdot a \cdot \Phi\bigl( n/b \bigr) \le a \Phi(n). \] In that case, the choice \(b=n^c\) would become affordable, and the above recurrence would be consistent with a polylogarithmic choice of \(\Phi(n)\). However, such an improvement is impossible in general. Indeed, the counterexample in Remark~\ref{Block averaging is sharp}, applied to the \(1/2\)-moment version of the same block-averaging argument, shows that the \(b^{1/2}\)-error is sharp in the BSA case. \end{remark}

\subsection{Phase transition for the estimate of the moments of the sensitivity of polynomial threshold function}
\label{PhaseTr-s}
Let $f\in PTF_d$ and let $s_f$ be its sensitivity function,  then
\begin{equation}
\label{PhaseTr}
\mathbb E[s_f(x)^\eta] \le \begin{cases} C(d) (\log n)^{C\,d^2}, \,\, \eta \in (0, 1/2],
\\
C(d) \, n^{\eta-1/2} \,(\log n)^{C\,d\, \log d}, \,\, \eta \in (1/2, 1]\,.
\end{cases}
\end{equation}

This can be derived easily from \cite{BSA_of_PTF}. The proof in Section \ref{balance} also gives exactly that (it is basically the same proof but with balanced partition).

But we call the attention of the reader to a strange fact, namely, that the proof by recursion provided in Section \ref{lowerMoments}  also gives \eqref{PhaseTr} with the exception of the case $\eta=1/2$.

\medskip

This looks strange and mysterious. The proofs are rather close in spirit after all.

\bigskip

{\bf Acknowledgement.} 
The authors acknowledge the use of OpenAI's ChatGPT (GPT-5.5 Thinking) as an auxiliary tool during the preparation of this work. In several instances, after the authors had developed mathematical ideas or proof strategies and verified them by hand at a preliminary level, the tool was used to generate draft arguments for comparison, refinement, and clarification. It was also used for mathematical discussion, assistance with calculations, notation, and exposition. All mathematical statements and proofs were independently checked by the authors, who take full responsibility for the content.

\bibliographystyle{plain}
\bibliography{mybibliography}

\end{document}